\documentclass[11pt]{amsart}
\usepackage[utf8]{inputenc}
\usepackage{graphicx,xcolor,verbatim}
\usepackage{amscd,amsmath,amsfonts,amssymb}
\usepackage{mathtools}
\usepackage[foot]{amsaddr}
\usepackage{cite}
\usepackage[pdfborder=000,pdftex=true]{hyperref}

\newtheorem{theorem}{Theorem}
\newtheorem{lemma}[theorem]{Lemma}
\newtheorem{proposition}[theorem]{Proposition}
\newtheorem{remark}[theorem]{Remark}

\newtheorem{definition}{Definition}

\numberwithin{equation}{section}

\newcommand{\R}{\mathbb R}
\newcommand{\N}{\mathbb N}

\newcommand{\SD}{\Sigma_{\mathcal D}}

\newcommand{\SN}{\Sigma_{\mathcal N}}

\begin{document}
	
\title[Nonlocal critical problems near resonance]{Nonlocal critical problems with mixed boundary conditions and nearly resonant perturbations}

\author{Eduardo Colorado}
\address[E. Colorado]{Departamento de Matem\'{a}ticas, Universidad Carlos III de Madrid\\
Av. Universidad 30, 28911 Legan\'{e}s (Madrid), Spain}
\email{\tt eduardo.colorado@uc3m.es}

\author{Giovanni Molica Bisci}
\address[G. Molica Bisci]{Department of Human Sciences and Promotion of Quality of Life, San Raffaele University, via di Val Cannuta 247, I-00166 Roma, Italy}
\email{\tt giovanni.molicabisci@uniroma5.it}

\author{Alejandro Ortega$^\dagger$}
\address[A. Ortega]{Dpto. de Matem\'aticas Fundamentales, Facultad de Ciencias, UNED, 28040 Madrid, Spain}
\email{\tt alejandro.ortega@mat.uned.es}

\author{Luca Vilasi}
\address[L. Vilasi]{Department of Mathematical and Computer Sciences, Physical Sciences and Earth Sciences\\
University of Messina\\
Viale F. Stagno d’Alcontres, 31 - 98166 Messina, Italy}
\email{\tt lvilasi@unime.it}

\keywords{Fractional Laplacian, Variational Methods, $\nabla$-Theorems, Mixed Boundary Data, Superlinear and Critical Nonlinearities.\\
\phantom{aa} 2010 AMS Subject Classification: Primary: 35R11, 35A15, 35S15, 49J35; Secondary: 35J61, 35B33, 58E05
\phantom{aa} $^\dagger$Corresponding author: A. Ortega.
}

\maketitle

\begin{center}
{\it   \small To Alonso, at the start of the journey}
\end{center}	

\begin{abstract}
We consider the following nonlocal critical problem with mixed Dirichlet-Neumann boundary conditions,
\begin{equation*}
        \left\{
        \begin{tabular}{lcl}
        $(-\Delta)^su=\lambda u+|u|^{2_s^*-2}u$ & &in $\Omega$, \\[2pt]
        $\mkern+51mu u=0$& &on $\Sigma_{\mathcal{D}}$, \\[2pt]
        $\mkern+36mu \displaystyle \frac{\partial u}{\partial \nu}=0$& &on $\Sigma_{\mathcal{N}}$,
        \end{tabular}
        \right.
\end{equation*}
where $(-\Delta)^s$, $s\in (1/2,1)$, is the spectral fractional Laplacian operator, $\Omega\subset\mathbb{R}^N$, $N>2s$, is a smooth bounded domain, $2_s^*=\frac{2N}{N-2s}$ denotes the critical fractional Sobolev exponent, $\lambda>0$ is a real parameter, $\nu$ is the outwards normal to $\partial\Omega$,  $\Sigma_{\mathcal{D}}$, $\Sigma_{\mathcal{N}}$ are smooth $(N-1)$--dimensional submanifolds of $\partial\Omega$ such that
$\Sigma_{\mathcal{D}}\cup\Sigma_{\mathcal{N}}=\partial\Omega$,
$\Sigma_{\mathcal{D}}\cap\Sigma_{\mathcal{N}}=\emptyset$ and
$\Sigma_{\mathcal{D}}\cap\overline{\Sigma}_{\mathcal{N}}=\Gamma$ is a smooth $(N-2)$--dimensional submanifold of $\partial\Omega$.
By employing a $\nabla$-theorem we prove the existence of multiple solutions when the parameter $\lambda$ is in a left neighborhood of a given  eigenvalue of $(-\Delta)^s$.
\end{abstract}

\section{Introduction}
In this paper we analyze the multiplicity of solutions to the following nonlocal problem
\begin{equation}\label{problem}\tag{$P_\lambda$}
\left\{
        \begin{tabular}{lcl}
        $(-\Delta)^su=\lambda u + |u|^{2_s^*-2}u$ & &in $\Omega$, \\[3pt]
        $\mkern+21muB(u)=0$& &on $\partial\Omega$,
        \end{tabular}
        \right.
\end{equation}
where $\Omega\subset\R^N$ is a bounded domain with smooth boundary, $N>2s$, $s\in(1/2,1)$, $2_s^*:=\frac{2N}{N-2s}$ denotes the critical fractional Sobolev exponent and $\lambda$ is a positive parameter. The fractionality range $\frac{1}{2} < s < 1$ is the correct one for mixed boundary problems due to the natural embedding of the associated functional space, see Remark \ref{rem:range_s}. Here, $(-\Delta)^s$ denotes the spectral fractional Laplace operator on $\Omega$ endowed with the mixed Dirichlet-Neumann ($\mathcal{D}$-$\mathcal{N}$ for short) boundary conditions
$$
B(u):=u\chi_{\SD} +\frac{\partial u}{\partial\nu}\chi_{\SN},
$$
where $\nu$ is the outward unit normal to $\partial\Omega$, $\chi_A$ denotes the characteristic function of the set $A\subset\partial\Omega$ and $\Omega$ satisfies the following set of assumptions:
\begin{itemize}
\item[$(\Omega_1)$] $\Omega\subset\R^N$ is a bounded Lipschitz domain;
\item[$(\Omega_2)$] $\SD$ and $\SN$ are smooth $(N-1)$--dimensional submanifolds of $\partial\Omega$;
\item[$(\Omega_3)$] $\SD$ is a closed manifold with positive $(N-1)$--dimensional Lebesgue measure, namely $|\SD|=\alpha\in(0,|\partial\Omega|)$;
\item[$(\Omega_4)$] $\SD\cap\SN=\emptyset$, $\SD\cup\SN=\partial\Omega$ and $\SD\cap\overline{\Sigma}_\mathcal{N}=\Gamma$, where $\Gamma$ is a smooth $(N-2)$--dimensional submanifold of $\partial\Omega$.
\end{itemize}

Denote by $|\Omega|$ the $N$-dimensional Lebesgue measure of $\Omega$ and by $\widetilde{S}(\Sigma_{\mathcal{D}})$ the fractional Sobolev constant for the mixed $\mathcal{D}$-$\mathcal{N}$ boundary data setting (see \eqref{Sob_mix}). Our main result reads as follows:
\begin{theorem}\label{mainresult}
Assume $(\Omega_1)-(\Omega_4)$ and let $\Lambda_k$, $k\geq 2$, be an eigenvalue of multiplicity $m\in\N$ of the problem
\begin{equation}\label{eigenproblem}
	\left\{
	\begin{array}{rl}
		(-\Delta)^{s} u = \lambda u  & \text{in } \Omega,\\
		B(u)=0\mkern+12mu & \text{on } \partial\Omega.
	\end{array}
	\right.
\end{equation}
Then, setting
\begin{equation}\label{defdeltak}
	\delta_k:=\min\left\lbrace \frac{\Lambda_k-\Lambda_{k-1}}{\Lambda_{k-1}+|\Omega|^{-\frac{2s}{N}}\widetilde{S}(\Sigma_D)}, \frac{\Lambda_{k+m}-\Lambda_k}{\Lambda_{k+m}}  \right\rbrace |\Omega|^{-\frac{2s}{N}}\widetilde{S}(\Sigma_{\mathcal{D}}), 
\end{equation}
for every $\lambda\in \left(\Lambda_k-{\delta_k},\Lambda_k\right)$, \eqref{problem} has at least two solutions. 
\end{theorem}

Contrary to nonlocal elliptic problems endowed with Dirichlet or Neumann boundary conditions, which have been extensively analyzed over the last two decades (see for instance the monograph \cite{MRS} and references therein), results about existence, multiplicity or qualitative properties of solutions to this type of problems with mixed boundary conditions are much less known. Recent contributions along this direction can be found in \cite{Barrios2020, Carmona2020a, carcolleoort2020, colort2019, Leonori2018, LopezOrtega2021, mov2023, Mukherjee2024, Mukherjee2025, Ortega2023}. In particular, in \cite{mov2023}, the authors proved a multiplicity result for a subcritical problem with nearly resonant perturbations by means of a so-called $\nabla$-theorem due to Marino and Saccon (cf. \cite{marsac1997some}). The main aim of this work is to extend these results to nonlinearities in the critical regime.
 
Although $\nabla$-theorems have been previously applied to obtain multiplicity of solutions (cf. \cite{mu0, mms, marsac1997some, mu2, mu1, mu4}), this is, up to our knowledge, the first work where this tool is used to deal with critical problems.  In this case, in addition to the usual geometric and compactness analysis, a careful analysis is required to obtain explicit estimates of some functional quantities related to \eqref{problem}. To be more precise, the key points of our approach are the following. After performing a spectral decomposition of the underlying energy space, we first determine the precise energy threshold $\varepsilon_\lambda$ below which one obtains only trivial critical points for the energy functional $I_\lambda$ associated with \eqref{problem} restricted to a certain subspace (see Lemma \ref{uniquecritpoint} below). Regarding this step, the harmonic mean of eigenvalues adjacent to the eigenvalue $\Lambda_k$ plays a crucial role (see Remark \ref{ra}), guaranteeing that $\varepsilon_\lambda$ is, in turn, lower than the compactness threshold $c^*$ of Palais-Smale sequences (being $c^*$ determined in Proposition \ref{propPS}). Next, we prove the compactness of sequences whose projections on the eigenspace associated with the eigenvalue $\Lambda_k$ and on its orthogonal satisfy certain conditions (see Lemma \ref{Ujbounded}). This step is crucial to prove the validity of the so-called $\nabla$-condition, a major ingredient of the $\nabla$-theorem (see Definition \ref{def_nabla_condition}). Lastly, we obtain a precise bound for $I_\lambda$ on the eigenspace spanned by the first $k+m-1$ eigenvalues, $m$ being the multiplicity of $\Lambda_k$ (see Lemma \ref{lemaa}). The combination of these results allows us to apply the aforementioned $\nabla$-theorem (see Theorem \ref{nablathm}) to deduce the existence of at least two solutions to \eqref{problem} for $\lambda$ close to every $\Lambda_k$, $k\geq 2$. 
It is worth mentioning that the range of $\lambda$ in Theorem \ref{mainresult} follows the lines of that obtained in \cite{Fiscella2016} by means of an abstract critical point theorem due to Bartolo, Benci and Fortunato (cf. \cite{Bartolo1983}).
\vspace{0.1cm}\\
The paper is organized as follows. In Section \ref{functionalsettings} we introduce the functional framework as well as the Sobolev constant associated with \eqref{problem} which, as it is customary when dealing with critical problems, plays a central role in the compactness properties of minimizing sequences for $I_\lambda$. In Section \ref{nablasect} we introduce the so-called $\nabla$-theorem, Theorem \ref{nablathm}, and prove all the technical results needed to apply it. Finally, in Section \ref{proofofmainresult}, we prove the main result of this paper, the aforementioned Theorem \ref{mainresult}.

\section{Functional framework}\label{functionalsettings}
We begin by first recalling the definition of the fractional Laplace operator, based on the spectral decomposition of the classical Laplace operator under mixed boundary conditions. Let $(\lambda_i,\varphi_i)$ be the eigenvalues and the eigenfunctions (normalized with respect to the $L^2(\Omega)$-norm) of $-\Delta$ with homogeneous mixed Dirichlet-Neumann boundary conditions, respectively. Then $(\lambda_i^s,\varphi_i)$ are the eigenvalues and the eigenfunctions of $(-\Delta)^s$. Consequently, given two smooth functions 
$$
u_i(x)=\sum_{j\geq1}\langle u_i,\varphi_j\rangle_{2}\varphi_j(x), \quad i=1,2,
$$
where $\langle u,v\rangle_2=\displaystyle\int_{\Omega}uv\,dx$ is the standard scalar product on $L^2(\Omega)$, one has
\begin{equation}\label{pre_prod}
\langle(-\Delta)^s u_1, u_2\rangle_{2} = \sum_{j\ge 1} \lambda_j^s\langle u_1,\varphi_j\rangle_{2} \langle u_2,\varphi_j\rangle_{2},
\end{equation}
that is, the action of the fractional operator on a smooth function $u_1$ is given by
\begin{equation*}
(-\Delta)^su_1=\sum_{j\ge 1} \lambda_j^s\langle u_1,\varphi_j\rangle_{2}\varphi_j.
\end{equation*}
As a consequence, the {\it spectral} fractional Laplace operator $(-\Delta)^s$ is well defined in the following Hilbert space of functions that vanish on $\SD$:
\begin{equation*}
H_{\Sigma_{\mathcal{D}}}^s(\Omega)\vcentcolon=\left\{u=\sum_{j\ge 1} a_j\varphi_j\in L^2(\Omega):\ u=0\ \text{on }\Sigma_{\mathcal{D}},\ \|u\|_{H_{\Sigma_{\mathcal{D}}}^s}^2:=
\sum_{j\ge 1} a_j^2\lambda_j^s<+\infty\right\}.
\end{equation*}
Thus, given $u\in H_{\SD}^s(\Omega)$, it follows by definition that
\begin{equation*}
\left\| u\right\|_{H_{\SD}^s}=\left\| (-\Delta)^{s/2} u\right\|_{2}.
\end{equation*}
Moreover, taking in mind \eqref{pre_prod}, the norm $\left\| \cdot\right\|_{H_{\SD}^s}$ is induced by the scalar product
\begin{equation*}
\langle u_1,u_2\rangle_{H_{\Sigma_{\mathcal{D}}}^s}=\left\langle (-\Delta)^su_1,u_2\right\rangle_{2}=\langle (-\Delta)^{\frac{s}{2}}u_1,(-\Delta)^{\frac{s}{2}}u_2\rangle_{2}=\left\langle u_1,(-\Delta)^su_2\right\rangle_{2},
\end{equation*}
for all $u_1,u_2\in H_{\SD}^s(\Omega)$. The above chain of identities can be simply stated as an integration-by-parts like formula.

\begin{remark}\label{rem:range_s}
{\rm As it is proved in \cite[Theorem 11.1]{Lions1972}, if $0<s\le \frac{1}{2}$, then $H_0^s(\Omega)=H^s(\Omega)$ and, thus, also
$H_{\Sigma_{\mathcal{D}}}^s(\Omega)=H^s(\Omega)$ while, for $\frac 12<s<1$, we have $H_0^s(\Omega)\subsetneq H^s(\Omega)$. Indeed, in the latter case the trace operator Tr$: H^s(\Omega)\to L^2(\partial\Omega)$ is a linear and continuous operator, see also\cite[Section 7.2]{BonforteSirVaz2015}. Therefore, the range $\frac 12<s<1$ ensures $H_{\Sigma_{\mathcal{D}}}^s(\Omega)\subsetneq H^s(\Omega)$ and it provides us with the appropriate functional space for the mixed boundary problem \eqref{problem}.
}
\end{remark}

As commented before, this spectral definition of $(-\Delta)^s$ allows us to integrate by parts in the proper spaces, so that a natural definition of weak solution to \eqref{problem} is the following.

\begin{definition}
{\rm	
We say that $u\in H_{\SD}^s(\Omega)$ is a weak solution to \eqref{problem} if, for all $v\in H_{\SD}^s(\Omega)$,
\begin{equation}\label{weak_solution}
\int_\Omega (-\Delta)^{s/2}u (-\Delta)^{s/2} v\, dx =\lambda\int_\Omega uv \; dx + \int_\Omega |u|^{2_s^*-2}uv\, dx,
\end{equation}
that is,
\begin{equation*}
	\langle u,v \rangle_{H_{\Sigma_{\mathcal{D}}}^s}=\langle \lambda u +|u|^{2_s^*-2}u, v\rangle_{2}.
\end{equation*}
}
\end{definition}
Let us note that \eqref{weak_solution} is well-defined because of the embedding $H_{\Sigma_{\mathcal{D}}}^s(\Omega)\hookrightarrow L^{2_s^*}(\Omega)$, so given $u\in H_{\Sigma_{\mathcal{D}}}^s(\Omega)$, we have $\lambda u+|u|^{2_s^*-2}u\in L^{\frac{2N}{N+2s}}\hookrightarrow \left(H_{\Sigma_{\mathcal{D}}}^s(\Omega)\right)'$.\\

\noindent The energy functional $I_\lambda:H_{\Sigma_{\mathcal{D}}}^s(\Omega)\to\R$ associated with \eqref{problem} is given by
\begin{equation}\label{functional_down}
I_\lambda(u):=\frac12\int_{\Omega}|(-\Delta)^{\frac{s}{2}}u|^{2}dx-\frac{\lambda}{2}\int_{\Omega}u^2dx-\frac{1}{2_s^*}\int_{\Omega}|u|^{2_s^*}dx.
\end{equation}
Actually, for all $u,v\in H_{\Sigma_{\mathcal{D}}}^s(\Omega)$, one has
\begin{equation*}
\langle I_\lambda'(u), v\rangle_{H^{-s}}  = \int_{\Omega}(-\Delta)^{\frac{s}{2}}u(-\Delta)^{\frac{s}{2}}v\,dx-\lambda\int_{\Omega}uv\,dx  - \int_{\Omega}|u|^{2_s^*-2}u\,v\,dx,
\end{equation*}
where $\langle\cdot,\cdot\rangle_{H^{-s}}$ denotes the duality product between $H_{\Sigma_{\mathcal{D}}}^s(\Omega)$ and the dual space $H^{-s}(\Omega)\vcentcolon=\left(H_{\Sigma_{\mathcal{D}}}^s(\Omega)\right)'$.

When one considers Dirichlet boundary conditions the 
fractional Sobolev inequality holds: for all $u\in H_{0}^s(\Omega)$ where $N>2s$ and $r\in[1,2^*_s]$, one has
\begin{equation}\label{sobolev}
  C\left(\int_{\Omega}|u|^rdx\right)^{\frac{2}{r}}\leq \int_{\Omega}\left|(-\Delta)^{\frac{s}2}u\right|^2dx.
\end{equation}
When $r=2_s^*$ the best constant in \eqref{sobolev}, 
denoted by $S(N,s)$, is independent of $\Omega$ and its exact value is given by
\begin{equation*}
S(N,s)=2^{2s}\pi^s\frac{\Gamma\left(\frac{N+2s}{2}\right)}{\Gamma\left(\frac{N-2s}{2}\right)}\left(\frac{\Gamma(\frac{N}{2})}{\Gamma(N)}\right)^{\frac{2s}{N}}.
\end{equation*}
Since it is not achieved in any bounded domain, one has
\begin{equation*}
S(N,s)\left(\int_{\mathbb{R}^N}|u|^{2_s^*}dx\right)^{\frac{2}{2_s^*}}\leq \int_{\mathbb{R}^{N}}|(-\Delta)^{\frac{s}2}u|^2dx,
\end{equation*}
for all $u\in H^s(\mathbb{R}^N)$. In the whole space, the latter inequality turns into equality for the family
\begin{equation*}
u_{\varepsilon}(x)=\frac{\varepsilon^{\frac{N-2s}{2}}}{(\varepsilon^2+|x|^2)^{\frac{N-2s}{2}}},
\end{equation*}
with arbitrary $\varepsilon>0$, (cf. \cite{bracoldepsan2013a}).

When mixed $\mathcal{D}$-$\mathcal{N}$ boundary conditions are considered the situation is quite similar since the Dirichlet condition is imposed on a set $\Sigma_{\mathcal{D}} \subset \partial \Omega$ such that $|\Sigma_{\mathcal{D}}|=\alpha\in(0,|\partial\Omega|)$.

\begin{definition}\label{defi_sob_const}
{\rm
The Sobolev constant relative to the Dirichlet boundary $\Sigma_{\mathcal{D}}$ is defined by
\begin{equation}\label{Sob_mix}
\widetilde{S}(\Sigma_{\mathcal{D}})=\inf_{\substack{u\in
H_{\Sigma_{\mathcal{D}}}^s(\Omega)\\ u\not\equiv
0}}\frac{\|u\|_{H_{\Sigma_{\mathcal{D}}}^s}^2}{\|u\|_{2_s^*}^2}.
\end{equation}
}
\end{definition}

Since $0<\alpha<|\partial\Omega|$ and $H_{0}^s(\Omega)\subsetneq H_{\Sigma_{\mathcal{D}}}^s(\Omega)$ we have
\begin{equation}\label{const}
0<\widetilde{S}(\Sigma_{\mathcal{D}})<S(N,s).
\end{equation}
Actually, by \cite[Proposition 3.6]{colort2019} we have $\widetilde{S}(\Sigma_{\mathcal{D}})\leq 2^{-\frac{2s}{N}}S(N,s)$ and, because of \cite[Theorem 2.9]{colort2019}, if $\widetilde{S}(\Sigma_{\mathcal{D}})<2^{-\frac{2s}{N}}S(N,s)$, then $\widetilde{S}(\Sigma_{\mathcal{D}})$ is attained. Due to \eqref{Sob_mix} and \eqref{const}, the following Sobolev-type inequality holds in the mixed boundary data framework.

\begin{lemma}\cite[Lemma 2.4]{colort2019}\label{lem:traceineq}
For all $u\in H_{\SD}^s(\Omega)$ it holds
\begin{equation}\label{mix_sobolev}
	\widetilde{S}(\Sigma_{\mathcal{D}})\left(\int_{\Omega}|u|^{2^*_s}dx\right)^{\frac{2}{2^*_s}}\leq \int_{\Omega}|(-\Delta)^{\frac{s}2}u|^2dx.
\end{equation}
\end{lemma}
In what follows, we shall use the same symbols $c,C$, or some subscripted version of them, to denote possibly different positive constants. Moreover, if $A\subseteq H_{\SD}^s(\Omega)$ and $r>0$, we will denote by $B_r(A)$ and $\partial B_r(A)$ the sets
\begin{equation*}
B_r(A):=\{u\in A: \left\|u\right\|_{H_{\SD}^s}\leq r\}\qquad\text{and}\qquad \partial B_r(A):=\{u\in A: \left\|u\right\|_{H_{\SD}^s}= r\}.
\end{equation*}
\section{Linking geometry, $\nabla$- and $(PS)$-conditions}\label{nablasect}
The main tool we rely upon for the proof of Theorem \ref{mainresult} is  a ``mixed-type" theorem, also known as a $\nabla$-theorem, due to Marino and Saccon, in which properties of both a functional and its gradient are used (\!\!\cite[Theorem 2.10]{marsac1997some}). Basically, a linking structure together with a suitable condition of the gradient on certain subsets allow us to deduce the existence of multiple critical points. Before stating this theorem, let us fix some notation and give some definitions.

Let $X$ be a Hilbert space and $M$ a closed subspace of $X$, we denote by $\Pi_M:X\to M$ the orthogonal projection of $X$ on $M$ and by $d(u,M)\vcentcolon=\inf_{v\in M}d(u,v)$ the distance of $u\in X$ from $M$. In addition, let $I:X\to\R$ be a $C^1$ functional and $a,b\in\R\cup\{-\infty,+\infty\}$. With these ingredients at hand, we give the following definition.

\begin{definition}\label{def_nabla_condition}
{\rm	
We say that $(I,M,a,b)$ satisfy the $\nabla$-condition, abbreviated by $(\nabla)(I,M,a,b)$, if there exists $\gamma>0$ such that
$$
\inf\Big\{\| \Pi_M\nabla I(u)\|_{X}: u\in M, \; a\leq I(u)\leq b, \;d(u,M)\leq\gamma \Big\} >0.
$$
}
\end{definition}

\noindent We also recall that $I$ is said to satisfy the Palais-Smale condition at level $c\in\R$, $(PS)_c$ for short, if any sequence $\{u_j\}\subset X$ satisfying
$$
I(u_j)\to c, \quad I'(u_j)\to 0 \text{ in } X', \quad \text{as } j\to\infty,
$$
has a convergent subsequence.

\begin{theorem}\label{nablathm}
Let $X$ be a Hilbert space and let $X_i$, $i=1,2,3$, be three subspaces of $X$ such that $X=X_1\oplus X_2 \oplus X_3$, with $dim \ X_i <\infty$ for $i=1,2$. Denote by $\Pi_i:X\to X_i$ the orthogonal projection of $X$ on $X_i$. Let $I:X\to\R$ be a $C^1$ functional, $\varrho,\varrho',\varrho'',\varrho_1$ be such that $\varrho_1>0$, $0\leq\varrho'<\varrho<\varrho''$, and define
\begin{align*}
\Delta&\vcentcolon=\left\lbrace u\in X_1\oplus X_2: \varrho'\leq \left\|\Pi_2 u \right\|\leq\varrho'', \left\| \Pi_1 u\right\|\leq\varrho_1\right\rbrace,\\
T&\vcentcolon=\partial_{X_1\oplus X_2}\Delta.
\end{align*}
Suppose that
\begin{equation}
a'\vcentcolon=\sup I(T) <\inf I(\partial B_\varrho(X_2\oplus X_3))=\vcentcolon a''.
\end{equation}
Let $a,b\in\R$ be such that $a'<a<a''$ and $b>\sup I(\Delta)$, and assume that $(\nabla)(I,X_1\oplus X_3,a,b)$ and $(PS)_c$ hold, the latter for all $c\in[a,b]$.

Then, $I$ has at least two critical points in $I^{-1}([a,b])$. Moreover, if in addition it is satisfied that $\inf I(B_\varrho(X_2\oplus X_3))>-\infty$ 
and $(PS)_c$ holds for all $c\in[a_1,b]$, with $a_1<\inf I(B_\varrho(X_2\oplus X_3))$ 
then $I$ has another critical level in $[a_1,a']$.
\end{theorem}

In order to use Theorem \ref{nablathm} to prove Theorem \ref{mainresult} we start by checking that the energy functional $I_\lambda$ defined in \eqref{functional_down} has the linking geometry stated in Theorem \ref{nablathm}.

Let $\{\Lambda_k\}$ be the eigenvalues of $(-\Delta)^s$ in $\Omega$ with mixed Dirichlet-Neumann boundary conditions, in increasing order and counted with their multiplicity,
$$
0<\Lambda_1<\Lambda_2\leq\ldots\Lambda_k\leq\Lambda_{k+1}\leq\ldots, \quad \Lambda_k\to +\infty \text{ as } k\to\infty,
$$
where, by definition, $\Lambda_k=\lambda_k^s$, $k\in\N$. For every $k\geq 1$, set
$$
\mathbb{H}_k\vcentcolon=\text{span}\{\varphi_1,\ldots,\varphi_k\}
$$
and
$$
\mathbb{P}_k\vcentcolon=\{u\in H_{\SD}^s(\Omega): \left\langle u,\varphi_j\right\rangle_{H_{\SD}^s}=0, \text{ for all } j=1,\ldots,k\},
$$
where $\varphi_k$ is the eigenfunction corresponding to $\Lambda_k$. In this setting, we have the following minimax characterization of the eigenvalues.
\begin{lemma}\cite[Lemma 6]{mov2023}\label{varcharacteigen}
For any $k\in\mathbb{N}$, one has
$$
\Lambda_k=\inf_{u\in \mathbb{P}_{k-1}}\frac{\left\|u\right\|_{H_{\SD}^s}^2}{\left\| u\right\|_{2}^2}=\sup_{u\in \mathbb{H}_{k}}\frac{\left\|u\right\|_{H_{\SD}^s}^2}{\left\| u\right\|_{2}^2}.
$$
\end{lemma}
		
We say that $\Lambda_k$, $k\geq 2$, has multiplicity $m\in\N$ if
$$
\Lambda_{k-1}<\Lambda_k=\Lambda_{k+1}=\ldots=\Lambda_{k+m-1}<\Lambda_{k+m};
$$
in this case the eigenspace associated with $\Lambda_k$ coincides with $\text{span}\{\varphi_k,\ldots,\varphi_{k+m-1}\}$. We set
\begin{equation*}
X_1\vcentcolon=\mathbb{H}_{k-1}, \quad X_2\vcentcolon=\text{span}\{\varphi_k,\ldots,\varphi_{k+m-1}\},\quad X_3\vcentcolon=\mathbb{P}_{k+m-1}.
\end{equation*}
\begin{proposition}\label{prop_supinf}
Let $k\in\N$, $k\geq 2$, and let $m\in\N$ be such that $\Lambda_{k-1}<\lambda<\Lambda_k=\ldots=\Lambda_{k+m-1}<\Lambda_{k+m}$. Then there exist $R,\varrho\in\R$, with $R>\varrho>0$, such that
\begin{equation*}
l_1\vcentcolon=\sup_{u\in B_R(X_1)\cup\partial B_R(X_1\oplus X_2) } I_\lambda(u) < \inf_{u\in \partial B_{\varrho}(X_2\oplus X_3)} I_\lambda(u)\vcentcolon=l_2.
	\end{equation*}
\end{proposition}

\begin{proof}
Given $u\in X_2\oplus X_3=\mathbb{P}_{k-1}$, by Lemma \ref{varcharacteigen} and the fractional Sobolev inequality \eqref{mix_sobolev} we get
\begin{align*}
I_\lambda(u) &= \frac{1}{2}\left\| u\right\|_{H_{\SD}^s}^2 -\frac{\lambda}{2}\left\|u\right\|_{2}^2-\frac{1}{2_s^*}\|u\|_{2_s^*}^{2_s^*}\\
&\geq \frac{1}{2}\left(1-\frac{\lambda}{\Lambda_k}\right)\left\| u\right\|_{H_{\SD}^s}^2-\frac{1}{2_s^* [\widetilde{S}(\Sigma_{\mathcal{D}})]^{\frac{2_s^*}{2}}}\|u\|_{H_{\SD}^s}^{2_s^*}\\
&=c_1\| u\|_{H_{\SD}^s}^2\left(1-c_2\left\| u\right\|_{H_{\SD}^s}^{2_s^*-2}\right)
\end{align*}
for suitable constants $c_1,c_2>0$. Then, if $u\in \partial B_{\varrho}(X_2\oplus X_3)$ with $\varrho\in\left(0,c_2^{-1/(2_s^*-2)}\right)$, we get
\begin{equation*}
I_\lambda(u)\geq c_1\varrho^2(1-c_2\varrho^{2_s^*-2})>0,
\end{equation*}
and, as a result, $l_2>0$. Now, let $u\in X_1=\mathbb{H}_{k-1}$, i.e.,
\begin{equation*}
u(x)=\sum_{i=1}^{k-1}\alpha_i\varphi_i(x),
\end{equation*}
with $\alpha_i=\langle u,\varphi_i\rangle_{2},\ i=1,2,\ldots,k-1$. By Lemma \ref{varcharacteigen} one has
\begin{equation}\label{diseqinHk-1}
\begin{split}	
I_\lambda(u) &\leq\frac{\Lambda_{k-1}}{2}\left\| u\right\|_{2}^2-\frac{\lambda}{2} \left\| u \right\|_{2}^2 -\frac{1}{2_s^*}\left\| u \right\|_{L^{2^*_s}(\Omega)}^{2^*_s}\leq \frac{\Lambda_{k-1}-\lambda}{2}\left\| u\right\|_{2}^2\\
&\leq 0,
\end{split}
\end{equation}
since $\Lambda_{k-1}<\lambda$. On the other hand, since
\begin{equation*}
I_\lambda(u)\leq \frac{1}{2}\left\| u\right\|_{H_{\SD}^s}^2 - \frac{1}{2_s^*}\|u\|_{2_s^*}^{2_s^*},
\end{equation*}
given $u\in X_1\oplus X_2=\mathbb{H}_k$ and taking into account that $\mathbb{H}_k$ is finite-dimensional, for suitable $c_3,c_4>0$, we get
\begin{equation*}
c_3 \left\| u\right\|_{H^s_{\SD}(\Omega)}\leq\|u\|_{2_s^*} \leq c_4 \left\| u\right\|_{H^s_{\SD}(\Omega)}.
\end{equation*}
We then get
\begin{equation}\label{ineqilambda}
I_\lambda(u)\leq\frac{1}{2}\left\| u\right\|_{H_{\SD}^s}^2 - \frac{c_3^{2_s^*}}{2_s^*}\left\| u\right\|_{H^s_{\SD}}^{2_s^*}.
\end{equation}
Taking $R>0$ large enough, we conclude that $I_\lambda(u)\leq 0$ for $u\in \partial B_R(X_1\oplus X_2)$. This, together with \eqref{diseqinHk-1} implies that $l_1\leq 0$.
\end{proof}

Now we address the fulfillment of the $\nabla$-condition for $I_\lambda$. This will require some preliminary steps.
First, we prove the $(PS)_c$ condition for suitable levels $c\in\mathbb{R}$.
\begin{proposition}\label{propPS}
Let $\lambda\in(0,+\infty)$. Then, the functional $I_\lambda$ satisfies the $(PS)_c$ condition at any level 
\begin{equation}\label{level}
c<c^*:=\frac{s}{N}[\widetilde{S}(\Sigma_{\mathcal{D}})]^{\frac{N}{2s}}.
\end{equation}
\end{proposition}

\begin{proof}
As usual, we start by proving first that $(PS)$ sequences are bounded and then we finish by proving that such sequences satisfy the $(PS)_c$ condition for suitable levels $c$.

Let $c\in\R$ satisfy \eqref{level} above and $\{u_j\}_{j\in\mathbb{N}}\subset H_{\SD}^s(\Omega)$ be a $(PS)_c$ sequence for $I_\lambda$, namely
\begin{equation}\label{PS:1}
I_\lambda(u_j)\to c, \quad I_\lambda'(u_j)\to 0 \text{ in } H^{-s}(\Omega), \quad \text{as } j\to\infty.
\end{equation}
From \eqref{PS:1}, it follows that there exists $\kappa>0$ such that 
\begin{equation}\label{PS:2} 
\max\left\lbrace \left| I_\lambda(u_j)\right|, \left|\left\langle I'_\lambda(u_j),\frac{u_j}{\|u_j\|_{H_{\Sigma_{\mathcal{D}}}^s}}\right\rangle_{H^{-s}} \right|\right\rbrace \leq \kappa.
\end{equation}
Hence,
\begin{equation}\label{PS:3}
I_\lambda(u_j)-\frac{1}{2}\left\langle  I'_\lambda(u_j),u_j\right\rangle_{H^{-s}}\leq \kappa \left( 1+\|u_j\|_{H_{\Sigma_{\mathcal{D}}}^s}\right).
\end{equation}
Moreover, since
\[I_\lambda(u_j)-\frac{1}{2}\left\langle I'_\lambda(u_j),u_j\right\rangle_{H^{-s}}=\left(\frac{1}{2}-\frac{1}{2_s^*}\right)\|u_j\|_{2_s^*}^{2_s^*}=\frac{s}{N}\|u_j\|_{2_s^*}^{2_s^*},\]
because of \eqref{PS:3}, we have
\begin{equation}\label{PS:4}
\|u_j\|_{2_s^*}^{2_s^*}\leq c_1\left( 1+\|u_j\|_{H_{\Sigma_{\mathcal{D}}}^s}\right) 
\end{equation}
for a suitable constant $c_1>0$. Next, by H\"{o}lder's inequality, we get
\[\|u_j\|_{2}^2\leq |\Omega|^{\frac{2s}{N}}\|u_j\|_{2_s^*}^{2}\leq c_1^{\frac{2}{2_s^*}}\left(1+\|u_j\|_{H_{\Sigma_{\mathcal{D}}}^s}\right)^{\frac{2}{2_s^*}}\]
so that
\begin{equation}\label{PS:5}
\|u_j\|_{2}^2\leq c_2\left(1+\|u_j\|_{H_{\Sigma_{\mathcal{D}}}^s}\right)
\end{equation}
for some positive constant $c_2$. Thus, by \eqref{PS:1}, \eqref{PS:4} and \eqref{PS:5} we find
\begin{equation*}
\begin{split}
\kappa&\geq  I_\lambda(u_j)\\
&=\frac{1}{2}\|u_j\|_{H_{\Sigma_{\mathcal{D}}}^s}^2-\frac{\lambda}{2}\|u_j\|_{2}^2-\frac{1}{2_s^*}\|u_j\|_{2_s^*}^{2_s^*}
\geq\frac{1}{2}\|u_j\|_{H_{\Sigma_{\mathcal{D}}}^s}^2-c_3(1+\|u_j\|_{H_{\Sigma_{\mathcal{D}}}^s})
\end{split}
\end{equation*}
for some constant $c_3>0$. Thus, 
\begin{equation}\label{boundd}
(PS)_c \text{ sequences are bounded in } H_{\Sigma_{\mathcal{D}}}^s(\Omega). 
\end{equation}
Next we address their compactness. Since $\{u_j\}_{j\in\mathbb{N}}$ is bounded in $H_{\Sigma_{\mathcal{D}}}^s(\Omega)$, there exists a subsequence (still denoted by) $\{u_j\}_{j\in\mathbb{N}}$ such that $u_j\rightharpoonup u_\infty$ for some $u_\infty\in H_{\SD}^s(\Omega)$, i.e., as $j\to \infty$,
\begin{equation}\label{PS:6}
\left\langle u_j,\varphi \right\rangle_{H_{\SD}^s}\to\left\langle u_\infty,\varphi\right\rangle_{H_{\SD}^s}, \quad \text{for all } \varphi\in H_{\SD}^s(\Omega).
\end{equation}
Furthermore, by \eqref{PS:4},  \eqref{boundd}, the embedding $H_{\Sigma_{\mathcal{D}}}^s(\Omega)\hookrightarrow L^{r}(\Omega)$ for $r\in[1,2_s^*]$, and the fact that $L^{2_s^*}(\Omega)$ is a reflexive space, we have that, up to a subsequence,
\begin{equation}\label{PS:7}
u_j\rightharpoonup u_\infty \quad \text{ in } L^{2_s^*}(\Omega),
\end{equation}
\begin{equation}\label{PS:8}
u_j\to u_\infty \quad\text{ in } L^{2}(\Omega),
\end{equation}
and
\begin{equation*}
u_j\to u_\infty\quad\text{ a.e. in } \mathbb{R}^N,
\end{equation*}
as $j\to \infty$. On the other hand, \eqref{boundd} together with \eqref{PS:4}, implies that $\|u_j\|_{2_s^*}$ is uniformly bounded in $j$ and thus, the sequence $\{|u_j|^{2_s^*-2}u_j\}_{j\in\mathbb{N}}$ is bounded in $L^{\frac{2_s^*}{2_s^*-1}}(\Omega)$. Therefore, by \eqref{PS:7}, we get
\begin{equation}\label{PS:10}
|u_j|^{2_s^*-2}u_j \rightharpoonup |u_\infty|^{2_s^*-2}u_\infty \quad \text{ in }\ L^{\frac{2_s^*}{2_s^*-1}}(\Omega),
\end{equation}
as $j\to \infty$, so that
\begin{equation*}
\int_{\Omega}|u_j|^{2_s^*-2}u_j\,\varphi dx\to\int_{\Omega} |u_\infty|^{2_s^*-2}u_\infty\,\varphi dx \quad \text{ for all } \varphi\in L^{2_s^*}(\Omega),
\end{equation*}
as $j\to \infty$. In particular, as $H_{\Sigma_{\mathcal{D}}}^s(\Omega)\subset L^{2_s^*}(\Omega)$,
\begin{equation}\label{PS:12}
\int_{\Omega}|u_j|^{2_s^*-2}u_j\,\varphi dx\to\int_{\Omega} |u_\infty|^{2_s^*-2}u_\infty\,\varphi dx \quad\text{ for all } \varphi\in H_{\Sigma_{\mathcal{D}}}^s(\Omega),
\end{equation}
as $j\to \infty$. Because of \eqref{PS:1}, for any $\varphi\in H_{\Sigma_{\mathcal{D}}}^s(\Omega)$,
\begin{equation*}
\langle I_\lambda'(u_j), \varphi\rangle_{H^{-s}}  = \int_{\Omega}(-\Delta)^{\frac{s}{2}}u_j(-\Delta)^{\frac{s}{2}}\varphi\,dx-\lambda\int_{\Omega}u_j\varphi\,dx  - \int_{\Omega}|u_j|^{2_s^*-2}u_j\,\varphi\,dx\to 0,
\end{equation*}
as $j\to \infty$. Thus, by \eqref{PS:6}, \eqref{PS:8} and \eqref{PS:12} we find
\[\int_{\Omega}(-\Delta)^{\frac{s}{2}}u_\infty(-\Delta)^{\frac{s}{2}}\varphi\,dx-\lambda\int_{\Omega}u_\infty\varphi\,dx  - \int_{\Omega}|u_\infty|^{2_s^*-2}u_\infty\,\varphi\,dx=0.\]
Thus, $u_\infty$ solves \eqref{problem}. On the other hand, using $u_\infty$ as a test function in \eqref{weak_solution} we get
\begin{equation}\label{PS:13}
I_\lambda(u_\infty)=\left(\frac{1}{2}-\frac{1}{2^*_s}\right)\int_\Omega |u_\infty|^{2^*_s} dx = \frac{s}{N}\int_{\Omega}|u_\infty|^{2_s^*}dx.
\end{equation}
Also, by Brezis-Lieb's Lemma, one has
\begin{equation*}	
\int_\Omega \left| (-\Delta)^{s/2} u_j\right|^2  dx  = \int_\Omega \left| (-\Delta)^{s/2} (u_j-u_\infty)\right|^2 dx + \int_\Omega \left| (-\Delta)^{s/2} u_\infty\right|^2  dx +o(1)
\end{equation*}
and
\begin{equation}\label{brezislieb2}
	\int_\Omega \left|u_j\right|^{2^*_s}  dx  = \int_\Omega \left|u_j-u_\infty\right|^{2^*_s}  dx + \int_\Omega \left|u_\infty\right|^{2^*_s}  dx +o(1),
\end{equation}
as $j\to\infty$. By \eqref{PS:8} and the last two identities we get
\begin{equation}\label{PS:14}
	\begin{split}
	I_\lambda(u_j) =& \frac{1}{2} \left(\int_\Omega \left| (-\Delta)^{s/2} (u_j-u_\infty)\right|^2 dx + \int_\Omega \left| (-\Delta)^{s/2} u_\infty\right|^2  dx \right) -\frac{\lambda}{2}\int_\Omega |u_\infty|^2 dx\\
	&-\frac{1}{2^*_s}\left( \int_\Omega \left|u_j-u_\infty\right|^{2^*_s}  dx + \int_\Omega \left|u_\infty\right|^{2^*_s}  dx \right) +o(1)\\
	 =& I_\lambda(u_\infty) +\frac{1}{2}\int_\Omega \left| (-\Delta)^{s/2} (u_j-u_\infty)\right|^2 dx - \frac{1}{2^*_s} \int_\Omega \left|u_j-u_\infty\right|^{2^*_s}  dx + o(1)
\end{split}
\end{equation}
as $j\to\infty$. Notice also that, due to \eqref{PS:7}, \eqref{PS:10} and \eqref{brezislieb2}, we obtain
\begin{equation}\label{A3}
	\begin{split}
\int_\Omega\left( |u_j|^{2^*_s-2}u_j - |u_\infty|^{2^*_s-2}u_\infty\right)\left( u_j-u_\infty\right)  dx	& =\int_\Omega |u_j|^{2^*_s} dx +\int_\Omega |u_\infty|^{2^*_s-2}u_\infty u_j dx\\
& \;\;\; -\int_\Omega |u_j|^{2^*_s-2} u_ju_\infty dx +\int_\Omega |u_\infty|^{2^*_s}dx\\
& =\int_\Omega |u_j|^{2^*_s}dx - \int_\Omega |u_\infty|^{2^*_s}dx + o(1)\\
& =\int_\Omega |u_j-u_\infty|^{2^*_s}dx +o(1)
\end{split}
\end{equation}
as $j\to\infty$. Even more, one has
\begin{equation}\label{relazIprimo}
\left\langle I'_\lambda(u_j), u_j-u_\infty\right\rangle_{H^{-s}} = \left\langle I'_\lambda(u_j)-I'_\lambda(u_\infty), u_j-u_\infty\right\rangle_{H^{-s}} =o(1)
\end{equation}
as $j\to\infty$ and, by \eqref{PS:8} and \eqref{A3},
\begin{equation}\label{A5}
	\begin{split}
	\left\langle I'_\lambda(u_j)-I'_\lambda(u_\infty),u_j-u_\infty\right\rangle_{H^{-s}} & =\int_\Omega \left| (-\Delta)^{s/2}(u_j-u_\infty)\right|^2 dx -\lambda\int_\Omega |u_j-u_\infty|^2 dx\\
	&\;\;\; -\int_\Omega \left(|u_j|^{2^*_s -2}u_j - |u_\infty|^{2^*_s -2}u_\infty \right)\left( u_j-u_\infty\right) dx\\
	& =\!\int_\Omega \left| (-\Delta)^{s/2}(u_j-u_\infty)\right|^2\! dx -\!\!\int_\Omega |u_j-u_\infty|^{2^*_s}dx +o(1)
	\end{split}
\end{equation}
as $j\to\infty$. Combining \eqref{relazIprimo} and \eqref{A5} we obtain
\begin{equation}\label{PS:15}
\int_{\Omega}|(-\Delta)^{\frac{s}{2}}(u_j-u_\infty)|^2dx=\int_{\Omega}|u_j-u_\infty|^{2_s^*}dx+o(1).
\end{equation} 
Since the sequence $\{\|u_j\|_{H_{\Sigma_{\mathcal{D}}}^s}\}_{j\in\mathbb{N}}$ is bounded in $\mathbb{R}$, up to a subsequence
\begin{equation}\label{PS:18}
\|u_j-u_\infty\|_{H_{\Sigma_{\mathcal{D}}}^s}^2=\int_{\Omega}|(-\Delta)^{\frac{s}{2}}(u_j-u_\infty)|^2dx\to\ell\geq 0,
\end{equation}
and therefore, by \eqref{PS:15},
\begin{equation}\label{PS:19}
\int_{\Omega}|u_j-u_\infty|^{2_s^*}dx\to\ell,
\end{equation}
as $j\to \infty$. Because of \eqref{PS:18} and \eqref{PS:19} we find $\ell^{\frac{2}{2_s^*}}\widetilde{S}(\Sigma_{\mathcal{D}})\leq \ell$.
Thus, either $\ell=0$ or $\ell\geq [\widetilde{S}(\Sigma_{\mathcal{D}})]^{\frac{N}{2s}}$. 
Now, by \eqref{PS:15}, we have
\begin{equation*}
\begin{split}
\frac{1}{2}\int_{\Omega}|(-\Delta)^{\frac{s}{2}}(u_j-u_\infty)|^2dx-\frac{1}{2_s^*}\int_{\Omega}|u_j-u_\infty|^{2_s^*}dx
=&\left(\frac{1}{2}-\frac{1}{2_s^*}\right)\int_{\Omega}|(-\Delta)^{\frac{s}{2}}(u_j-u_\infty)|^2dx+o(1)\\
=&\frac{s}{N}\int_{\Omega}|(-\Delta)^{\frac{s}{2}}(u_j-u_\infty)|^2dx+o(1),
\end{split}
\end{equation*}
which, combined with \eqref{PS:14}, produces
\begin{equation}\label{PS:16}
I_\lambda(u_j)=I_\lambda(u_\infty)+\frac{s}{N}\int_{\Omega}|(-\Delta)^{\frac{s}{2}}(u_j-u_\infty)|^2dx+o(1),
\end{equation}
as $j\to \infty$. Thus, thanks to \eqref{PS:1} and \eqref{PS:16} above we conclude
\begin{equation}\label{PS:17}
c=I_\lambda(u_\infty)+\frac{s}{N}\int_{\Omega}|(-\Delta)^{\frac{s}{2}}(u_j-u_\infty)|^2dx+o(1),
\end{equation}
as $j\to \infty$. Because of \eqref{PS:13}, \eqref{PS:18} and \eqref{PS:17} we get 
\begin{equation}\label{PS:20}
c=I_\lambda(u_\infty)+\frac{s}{N}\ell\geq\frac{s}{N}\ell.
\end{equation}
In case of having $\ell\geq [\widetilde{S}(\Sigma_{\mathcal{D}})]^{\frac{N}{2s}}$, by \eqref{PS:20} above, it follows that
\[ c\geq \frac{s}{N}\ell\geq\frac{s}{N}[\widetilde{S}(\Sigma_{\mathcal{D}})]^{\frac{N}{2s}},\]
in contradiction with \eqref{level}. We conclude then $\ell=0$ so that, by \eqref{PS:18}, we get
\[\|u_j-u_\infty\|_{H_{\Sigma_{\mathcal{D}}}^s}^2\to0\qquad\text{as}\ j\to+\infty,\] 
and the proof is finished.
\end{proof}

Now, let us set 
\begin{equation*}
\mathbb{E}_{k,m}:=\text{\rm span}\{\varphi_k,\ldots,\varphi_{k+m-1}\}\quad\text{and}\quad\mathbb V_{k,m}:=\mathbb H_{k-1}\oplus \mathbb{P}_{k+m-1},
\end{equation*}
so that $\mathbb V_{k,m}^\perp=\mathbb{E}_{k,m}$ and $H_{\SD}^s(\Omega)=\mathbb V_{k,m}\oplus\mathbb{E}_{k,m}$, and the corresponding projection operators	
\begin{equation*}
\Pi_1\vcentcolon=\Pi_{\mathbb{E}_{k,m}}\quad\text{and}\quad\Pi_2\vcentcolon=\Pi_{\mathbb V_{k,m}}.
\end{equation*}

\begin{lemma}\label{uniquecritpoint}
Let $k\in\N$, $k\geq 2$, and let $\Lambda_k$ be an eigenvalue of multiplicity $m\in\N$. Then, for any $\lambda\in(\Lambda_{k-1},\Lambda_{k+m})$, setting
\begin{equation}\label{defepsilonlambda}
\varepsilon_\lambda:=\left(\min\left\{\frac{\lambda-\Lambda_{k-1}}{\Lambda_{k-1}},\frac{\Lambda_{k+m}-\lambda}{\Lambda_{k+m}}\right\}\right)^{\frac{N}{2s}}c^*,
\end{equation}
with $c^*$ being defined in \eqref{level}, the unique critical point $u$ of ${I_\lambda}_{|\mathbb{V}_{k,m}}$ with $I_\lambda(u)\in(-\varepsilon_\lambda,\varepsilon_\lambda)$ is the trivial one.
\end{lemma}

\begin{proof}
By contradiction, assume that there exists $\bar\lambda\in(\Lambda_{k-1},\Lambda_{k+m})$ and $u_{\bar\lambda}\in \mathbb{V}_{k,m}\setminus\{0\}$, such that
\begin{equation}\label{relazJmuj}
\begin{split}
	&\left\langle I_{\bar\lambda}'(u_{\bar\lambda}),\psi\right\rangle_{H^{-s}}=0 \quad \text{for any } \psi\in \mathbb{V}_{k,m},\\
	&I_{\bar\lambda}(u_{\bar\lambda})\in(-\varepsilon_{\bar\lambda},\varepsilon_{\bar\lambda}).
\end{split}
\end{equation}
Testing the first equation of \eqref{relazJmuj} with $\psi=u_{\bar\lambda}$ it follows that
\begin{align*}
0 & = \left\| u_{\bar\lambda}\right\|_{H_{\SD}^s}^2 -\bar\lambda \left\| u_{\bar\lambda}\right\|_{2}^2-\left\| u_{\bar\lambda}\right\|_{2_s^*}^{2_s^*} =2I_{\bar\lambda}(u_{\bar\lambda}) +\frac{2}{2_s^*}\left\| u_{\bar\lambda}\right\|_{2_s^*}^{2_s^*}-\left\| u_{\bar\lambda}\right\|_{2_s^*}^{2_s^*}
\end{align*}
from which we conclude that
\begin{equation}\label{intFxUj}
\left\| u_{\bar\lambda}\right\|_{2_s^*}^{2_s^*}=\frac{N}{s}I_{\bar\lambda}(u_{\bar\lambda}).
\end{equation}
Now, let us write $u_{\bar\lambda}=u_{\bar\lambda}^{\mathbb{H}}+u_{\bar\lambda}^{\mathbb{P}}$, with $u_{\bar\lambda}^{\mathbb{H}}\in \mathbb{H}_{k-1}$ and $u_{\bar\lambda}^{\mathbb{P}}\in \mathbb{P}_{k+m-1}$. Taking $\psi=u_{\bar\lambda}^{\mathbb{H}}-u_{\bar\lambda}^{\mathbb{P}}$ in the first equation of \eqref{relazJmuj} and using Lemma \ref{varcharacteigen} we obtain
\begin{equation}\label{deltalambdakm}
	\begin{split}
	\int_\Omega |u_{\bar\lambda}|^{2_s^*-2}u_{\bar\lambda}\,(u_{\bar\lambda}^{\mathbb{H}}-u_{\bar\lambda}^{\mathbb{P}})dx & =\left\langle u_{\bar\lambda},u_{\bar\lambda}^{\mathbb{H}} \right\rangle_{H_{\SD}^s} - \left\langle u_{\bar\lambda},u_{\bar\lambda}^{\mathbb{P}} \right\rangle_{H_{\SD}^s} -\bar\lambda\int_\Omega u_{\bar\lambda}\left(u_{\bar\lambda}^{\mathbb{H}}-u_{\bar\lambda}^{\mathbb{P}}\right)dx\\
	& =\left\| u_{\bar\lambda}^{\mathbb{H}}\right\|_{H_{\SD}^s}^2-\left\| u_{\bar\lambda}^{\mathbb{P}}\right\|_{H_{\SD}^s}^2 \mkern-5mu -\bar\lambda\left\|u_{\bar\lambda}^{\mathbb{H}} \right\|_{2}^2 +\bar\lambda \left\|u_{\bar\lambda}^{\mathbb{P}} \right\|_{2}^2\\
	&\leq\left( 1-\frac{\bar\lambda}{\Lambda_{k-1}}\right)\left\| u_{\bar\lambda}^{\mathbb{H}}\right\|_{H_{\SD}^s}^2 \mkern-10mu + \left(\frac{\bar\lambda}{\Lambda_{k+m}}-1\right)\left\| u_{\bar\lambda}^{\mathbb{P}}\right\|_{H_{\SD}^s}^2 \\
	&\leq -\min\left\{\frac{\bar\lambda-\Lambda_{k-1}}{\Lambda_{k-1}},\frac{\Lambda_{k+m}-\bar\lambda}{\Lambda_{k+m}}\right\}\left\| u_{\bar\lambda}\right\|_{H_{\SD}^s}^2\\
	&=-\mathfrak{m_{\bar\lambda}}\left\| u_{\bar\lambda}\right\|_{H_{\SD}^s}^2,
\end{split}
\end{equation}	
where $\mathfrak{m_{\bar\lambda}}:=\min\left\{\frac{\bar\lambda-\Lambda_{k-1}}{\Lambda_{k-1}},\frac{\Lambda_{k+m}-\bar\lambda}{\Lambda_{k+m}}\right\}$ and we used the fact that $u_{\bar\lambda}^{\mathbb{H}}\perp u_{\bar\lambda}^\mathbb{P}$. On the other hand, because of H\"{o}lder's inequality and the Sobolev inequality \eqref{mix_sobolev} respectively, we obtain
\begin{equation}\label{ujtwostars}
\begin{split}	
	\left| \int_\Omega |u_{\bar\lambda}|^{2_s^*-2}u_j\,(u_{\bar\lambda}^{\mathbb{H}}-u_{\bar\lambda}^{\mathbb{P}})dx\right| & \leq \left\| |u_{\bar\lambda}|^{2_s^*-2}u_{\bar\lambda}\right\|_{\frac{2_s^*}{2_s^*-1}}\left\|u_{\bar\lambda}^{\mathbb{H}}-u_{\bar\lambda}^{\mathbb{P}} \right\|_{2_s^*}\\
    &  \leq \widetilde{S}(\Sigma_D)^{-\frac{1}{2}}\left\|u_{\bar\lambda}\right\|_{2_s^*}^{2^*_s-1}\left\|u_{\bar\lambda} \right\|_{H_{\SD}^s}.
\end{split}	
\end{equation}
Since $u_{\bar\lambda}\neq 0$, because of \eqref{mix_sobolev}, \eqref{deltalambdakm} and \eqref{ujtwostars}, we obtain
\begin{equation*}
\mathfrak{m_{\bar\lambda}}\widetilde{S}(\Sigma_D)^{\frac{1}{2}} \left\|u_{\bar\lambda}\right\|_{2_s^*}\leq\mathfrak{m_{\bar\lambda}}\left\| u_{\bar\lambda}\right\|_{H_{\SD}^s} \leq \widetilde{S}(\Sigma_D)^{-\frac{1}{2}} \left\|u_{\bar\lambda}\right\|_{2_s^*}^{2^*_s-1},
\end{equation*}
so that, by \eqref{intFxUj},
\begin{equation*}
\mathfrak{m_\lambda}\widetilde{S}(\Sigma_D)\leq\|u_{\bar\lambda}\|_{2_s^*}^{2_s^*-2} <\left(\frac{N}{s}\varepsilon_{\bar\lambda}\right)^{\frac{2s}{N}}=\mathfrak{m_{\bar\lambda}}\widetilde{S}(\Sigma_D),
\end{equation*}
clearly a contradiction.
\end{proof}

\begin{remark}\label{ra}
{\rm Setting $H_k^*:=HM(\Lambda_{k-1},\Lambda_{k+m})$ the harmonic mean of $\Lambda_{k-1}$ and $\Lambda_{k+m}$, it is easily verified that
\begin{equation}\label{disq}
        \varepsilon_\lambda=\left\{
        \begin{array}{ll}
        \left( \dfrac{\lambda-\Lambda_{k-1}}{\Lambda_{k-1}}\right) ^\frac{N}{2s} c^* & \textit{ if }\Lambda_{k-1}<\lambda\leq H_{k}^*, \\[2pt]
        \left( \dfrac{\Lambda_{k+m}-\lambda}{\Lambda_{k+m}}\right)^\frac{N}{2s} c^* & \text{ if } H_k^*<\lambda<\Lambda_{k+m}.
        \end{array}
        \right.
\end{equation}
As a result, for any $\lambda\in(\Lambda_{k-1},\Lambda_{k+m})$, one has 
\begin{equation}\label{llv1}
\varepsilon_\lambda\leq\varepsilon_{H_{k}^*}:=\left(\frac{\Lambda_{k+m}-\Lambda_{k-1}}{\Lambda_{k+m}+\Lambda_{k-1}}\right)^{\frac{N}{2s}}c^*<c^*.
\end{equation}
}
\end{remark}

\begin{lemma}\label{Ujbounded}
	Assume $k\in\N$, $k\geq 2$, and let  $\Lambda_k$ be an eigenvalue of multiplicity $m\in\N$, $\lambda\in\R$. Then, any sequence $\{u_j\}\subset H_{\SD}^s(\Omega)$ such that
	\begin{itemize}
		\item[$(i)$] $\sup_{j\in\N}|I_\lambda(u_j)|<+\infty$,
		\item[$(ii)$] $\Pi_1 [u_j]\to 0$ as $j\to\infty$ in $H_{\SD}^s(\Omega)$,
		\item[$(iii)$] $\Pi_2[\nabla I_\lambda(u_j)]\to 0$ as $j\to\infty$ in $H_{\SD}^s(\Omega)$,
	\end{itemize}
is bounded in $H_{\SD}^s(\Omega)$. Moreover, if we replace $(i)$ by 
\begin{itemize}
		\item[$(i^*)$] $\sup_{j\in\N}|I_\lambda(u_j)|<c^*$,
	\end{itemize}
then, up to a subsequence, $u_j$ strongly converges in $H_{\SD}^s(\Omega)$.
\end{lemma}

\begin{proof}
Arguing by contradiction and without losing generality, let $u_j=\Pi_1[u_j]+\Pi_2[u_j]$ satisfy $(i)-(iii)$ such that $\left\| u_j\right\|_{H_{\SD}^s}\to +\infty $ as $j\to+\infty$ and 
\begin{equation}\label{convergencesUj}
\frac{u_j}{\left\| u_j\right\|_{H_{\SD}^s}}\rightharpoonup u_\infty \text { in } H_{\SD}^s(\Omega), \quad \frac{u_j}{\left\| u_j\right\|_{H_{\SD}^s}}\to u_\infty \text { in } L^2(\Omega), 
\end{equation}
for some $u_\infty\in H_{\SD}^s(\Omega)$. Being $H_{\SD}^s(\Omega)$ a Hilbert space and $I_\lambda$ a $C^1$ functional, we define the gradient $\nabla I_\lambda$ of $I_\lambda$, as usual, by
\begin{equation}\label{dd}
\langle \nabla I_\lambda(u),v \rangle_{H_{\SD}^s}\vcentcolon=\langle I_\lambda'(u),v \rangle_{H^{-s}},
\end{equation}
for any $u,v\in H_{\SD}^s(\Omega)$. Let $r\in[1,2_s^*]$ and $\mathcal{K}: L^{\frac{r}{r-1}}(\Omega)\to H_{\SD}^s(\Omega)$ be the operator defined by $\mathcal{K}=(-\Delta)^{-s}$, that is, $\mathcal{K}(g)=v$, where $v\in H_{\SD}^s(\Omega)$ is a weak solution to
\begin{equation*}
\left\{
	\begin{array}{rl}
		(-\Delta)^{s} v = g  & \text{in } \Omega\\
		               B(v) = 0  & \text{on } \partial\Omega.
	\end{array}
	\right.
\end{equation*}
Let us note that,
\begin{align*}
	\left\langle \mathcal K\left( \lambda u+|u|^{2_s^*-2}u\right),\psi\right\rangle_{H_{\SD}^s} & =\int_\Omega \left(\lambda u+|u|^{2_s^*-2}u\right)\psi dx\\
	& =-\left\langle I'_\lambda(u),\psi\right\rangle_{H^{-s}} + \left\langle u,\psi\right\rangle_{H_{\SD}^s} \\
	&=   -\left\langle \nabla I_\lambda(u),\psi\right\rangle_{H_{\SD}^s} + \left\langle u,\psi\right\rangle_{H_{\SD}^s},
\end{align*}
and, hence
\begin{equation}\label{graduK}
\nabla I_\lambda(u)=u-\mathcal K(\lambda u+|u|^{2_s^*-2}u),
\end{equation}
for all $u\in H_{\SD}^s(\Omega)$. As a consequence, we find
\begin{align*}
	\left\langle \Pi_2\left[\nabla I_\lambda(u_j)\right],u_j\right\rangle_{H_{\SD}^s}\!\!  & = \left\langle \nabla I_\lambda(u_j),u_j\right\rangle_{H_{\SD}^s} - \left\langle \Pi_1\left[\nabla I_\lambda(u_j)\right],u_j\right\rangle_{H_{\SD}^s}\\
	&=\left\| u_j\right\|_{H_{\SD}^s}^2\!\!\!\! - \lambda \left\| u_j\right\|_{2}^2 -\!\!\int_\Omega |u_j|^{2_s^*} dx - \left\langle \Pi_1\left[u_j-\mathcal K(\lambda u_j+|u_j|^{2_s^*-2}u_j)\right]\!,u_j\right\rangle_{H_{\SD}^s}\!\!\!.
\end{align*}
Moreover, since $\left\langle \Pi_1[u],v\right\rangle_{H_{\SD}^s}=\left\langle u,\Pi_1[v]\right\rangle_{H_{\SD}^s}$ for all $u,v\in H_{\SD}^s(\Omega)$, one has
\begin{align*}
\left\langle \Pi_1[u_j-\mathcal K(\lambda u_j+|u_j|^{2_s^*-2}u_j)],u_j\right\rangle_{H_{\SD}^s} &= \left\| \Pi_1u_j\right\|_{H_{\SD}^s}^2 -\lambda \left\langle \Pi_1[u_j],\mathcal K(u_j) \right\rangle_{H_{\SD}^s}\\
&\ \ - \left\langle \Pi_1[u_j],\mathcal K(|u_j|^{2_s^*-2}u_j)\right\rangle_{H_{\SD}^s},
\end{align*}
and, by \eqref{graduK},
\begin{align*}
	\lambda \left\langle \Pi_1[u_j],\mathcal{K}(u_j)\right\rangle_{H_{\SD}^s}&+ \left\langle \Pi_1[u_j],\mathcal{K}(|u_j|^{2_s^*-2}u_j)\right\rangle_{H_{\SD}^s} = \lambda \left\| \Pi_1[u_j]\right\|_{2}^2 + \int_\Omega |u_j|^{2_s^*-2}u_j\Pi_1[u_j] dx.
\end{align*}
Therefore, we deduce
\begin{equation}\label{relationQnabla}
	\begin{split}
		\left\langle \Pi_2\left[\nabla I_\lambda(u_j)\right],u_j\right\rangle_{H_{\SD}^s} &=2I_\lambda(u_j) +\frac{N-2s}{N}\int_\Omega |u_j|^{2_s^*} dx  -\int_\Omega |u_j|^{2_s^*-2}u_j\Big(u_j-\Pi_1[u_j]\Big) dx\\
		&\ \  \ - \left\| \Pi_1[u_j]\right\|_{H_{\SD}^s}^2 +\lambda\left\| \Pi_1[u_j]\right\|_{2}^2.
	\end{split}
\end{equation}
Thus, by $(i)-(iii)$ it follows that
\begin{equation}\label{fractiongoingtozero}
\frac{\displaystyle -\frac{2s}{N}\int_\Omega |u_j|^{2_s^*} dx  +\int_\Omega |u_j|^{2_s^*-2}u_j\,\Pi_1[u_j] dx}{\left\| u_j\right\|_{H_{\SD}^s}^{2_s^*}}\to 0
\end{equation}
as $j\to \infty$.  Next, since the eigenfunctions of \eqref{eigenproblem} are bounded (cf. \cite[Proposition 7]{mov2023}), we obtain
\begin{equation*}
\frac{\displaystyle\int_\Omega\left| |u_j|^{2_s^*-2}u_j\,\Pi_1[u_j]\right| dx}{\left\| u_j\right\|_{H_{\SD}^s}^{2_s^*}}  \leq \frac{\left\| \Pi_1[u_j]\right\|_{\infty}\left\| u_j\right\|_{L^{2^*_s-1}(\Omega)}^{2_s^*-1}}{\left\| u_j\right\|_{H_{\SD}^s}^{2_s^*}} \leq c\frac{\left\| \Pi_1[u_j]\right\|_{\infty}}{\left\| u_j\right\|_{H_{\SD}^s}},
\end{equation*}
for some $c>0$, and hence, since $(ii)$ forces $\left\| \Pi_1[u_j]\right\|_{\infty}\to 0$ (on $\text{span}\{\varphi_k,\ldots,\varphi_{k+m-1}\}$ all norms are equivalent), we get
\begin{equation*}
\frac{\displaystyle \int_\Omega |u_j|^{2_s^*-2}u_j \Pi_1[u_j] dx}{\left\| u_j\right\|_{H_{\SD}^s}^{2_s^*}}\to 0, \quad \text{as }j\to\infty,
\end{equation*}
so that, taking in mind \eqref{fractiongoingtozero}, it follows that
\begin{equation}\label{convzeroratio}
\frac{\displaystyle \int_\Omega |u_j|^{2_s^*}dx}{\left\| u_j\right\|_{H_{\SD}^s}^{2_s^*}} \to 0, \quad \text{as } j\to \infty.
\end{equation}
Since
$$
\left\| u_j\right\|_{2}\leq |\Omega|^\frac{s}{N}\left\|u_j \right\|_{L^{2^*_s}(\Omega)}, \quad \text{for every } j\in\N, 
$$
by \eqref{convzeroratio} we deduce that
$$
\frac{\left\| u_j\right\|_{2}}{\left\| u_j\right\|_{H_{\SD}^s}}\to 0 \quad \text{as } j\to \infty,
$$
which, together with \eqref{convergencesUj}, forces $u_\infty\equiv 0$. As a result, we get
\begin{equation*}
\frac{I_\lambda(u_j)}{\left\|u_j \right\|_{H_{\SD}^s}^2}=\frac12-\frac{\lambda}{2} \frac{\left\|u_j \right\|_{2}^2}{ \left\|u_j \right\|_{H_{\SD}^s}^2}-\frac{1}{2_s^*}\frac{\displaystyle \int_\Omega |u_j|^{2_s^*}dx}{\left\|u_j \right\|_{H_{\SD}^s}^2}\to 0,
\end{equation*}
that is,
\begin{equation}\label{convergencetoNN-2s}
\frac{\left\|u_j\right\|^{2_s^*}_{L^{2^*_s}(\Omega)}}{\left\|u_j \right\|_{H_{\SD}^s}^2} \to \frac{2_s^*}{2} \quad \text{as } j\to \infty.
\end{equation}
By retracing the same steps as before and bearing also in mind that, by \eqref{convergencetoNN-2s}, the sequence $\left\lbrace \frac{\left\|u_j\right\|^{2_s^*}_{2^*_s}}{\left\|u_j \right\|_{H_{\SD}^s}^2}\right\rbrace$ is bounded in $\R$, we obtain
\begin{align*}
\frac{\displaystyle\int_\Omega\left| |u_j|^{2_s^*-2}u_j\,\Pi_1[u_j]\right| dx}{\left\| u_j\right\|_{H_{\SD}^s}^2}  &\leq \frac{\left\| \Pi_1[u_j]\right\|_{\infty}\left\| u_j\right\|_{2^*_s-1}^{2_s^*-1}}{\left\| u_j\right\|_{H_{\SD}^s}^2} \leq c_1\frac{\left\| \Pi_1[u_j]\right\|_{\infty}\left\| u_j\right\|_{2^*_s}^{2_s^*-1}}{\left\| u_j\right\|_{H_{\SD}^s}^2}\\
& \leq c_2\frac{\left\| \Pi_1[u_j]\right\|_{\infty}\left\| u_j\right\|_{H_{\SD}^s}^\frac{2(2_s^*-1)}{2^*_s}}{\left\| u_j\right\|_{H_{\SD}^s}^2} = c_2\frac{\left\| \Pi_1[u_j]\right\|_{\infty}}{\left\| u_j\right\|_{H_{\SD}^s}^\frac{2}{2^*_s}}\to 0,
\end{align*}
as $j\to \infty$. As a result, dividing both sides of \eqref{relationQnabla} by $\left\| u_j\right\|_{H_{\SD}^s}^2$ and taking also into account $(i)-(iii)$, we get
\begin{equation}
\frac{\left\|u_j\right\|^{2_s^*}_{2^*_s}}{\left\| u_j\right\|_{H_{\SD}^s}^2}\to 0\quad\text{as}\ j\to \infty,
\end{equation}
which contradicts \eqref{convergencetoNN-2s}. Hence $\{u_j\}$ is bounded in $H_{\SD}^s(\Omega)$. We now assume that 
\begin{equation}\label{n0}
|I_\lambda(u_j)|\leq c<c^*
\end{equation}
for some $c>0$ and for all $j\in\N$. Note that, because of $(ii)$, one has
\begin{equation}\label{n1}
\|\Pi_1[u_j]\|_{\infty}\to0\quad\text{as}\ j\to\infty.
\end{equation}
On the other hand,
\begin{equation*}
\begin{split}
\nabla I_{\lambda}\left(\Pi_2[u_j]\right)
&=\Pi_2[u_j]-\lambda\mathcal{K}\left(\Pi_2[u_j]\right) -\mathcal{K}\left(|\Pi_2\left[u_j\right]|^{2_s^*-2}\Pi_2[u_j]\right)
\end{split}
\end{equation*}
and
\begin{equation*}
\Pi_2\left[\nabla I_{\lambda}(u_j)\right]=\Pi_2[u_j]-\lambda\Pi_2\left[\mathcal{K}( u_j)\right]-\Pi_2\left[\mathcal{K}(|u_j|^{2_s^*-2}u_j)\right],
\end{equation*}
so that
\begin{equation*}
\begin{split}
\nabla I_{\lambda}(\Pi_2[u_j])-\Pi_2\left[\nabla I_{\lambda}(u_j)\right]=&\lambda\left(\Pi_2\left[\mathcal{K}( u_j)\right]-\mathcal{K}\left(\Pi_2[u_j]\right)\right)\\
&+\Pi_2\left[\mathcal{K}(|u_j|^{2_s^*-2}u_j)\right]-\mathcal{K}\left(|\Pi_2\left[u_j\right]|^{2_s^*-2}\Pi_2[u_j]\right)\\
=&\lambda\Pi_2\left[\mathcal{K}( u_j)-\mathcal{K}\left(\Pi_2[u_j]\right)\right]\\
&+\Pi_2\left[\mathcal{K}(|u_j|^{2_s^*-2}u_j)-\mathcal{K}\left(|\Pi_2\left[u_j\right]|^{2_s^*-2}\Pi_2[u_j]\right)\right]\\
=&\lambda\Pi_2\left[\mathcal{K}\left(u_j-\Pi_2[u_j]\right)\right]\\
&+\Pi_2\left[\mathcal{K}\left(|u_j|^{2_s^*-2}u_j-|\Pi_2\left[u_j\right]|^{2_s^*-2}\Pi_2[u_j]\right)\right]\\
=&\lambda\Pi_2\left[\mathcal{K}\left(\Pi_1[u_j]\right)\right]+\Pi_2\left[\mathcal{K}\left(|u_j|^{2_s^*-2}u_j-|\Pi_2\left[u_j\right]|^{2_s^*-2}\Pi_2[u_j]\right)\right].\\
\end{split}
\end{equation*}
Observe that, 
\begin{equation*}
\Pi_2\left[\mathcal{K}\left(\Pi_1[u_j]\right)\right]=\Pi_2\left[\mathcal{K}\left(\sum_{i=k}^{k+m-1}c_{i}\varphi_{i}\right)\right]=\Pi_2\left[\frac{1}{\Lambda_k}\sum_{i=k}^{k+m-1}c_{i}\varphi_{i}\right]=\frac{1}{\Lambda_k}\Pi_2\left[\Pi_1[u_j]\right]=0,
\end{equation*}
so, we get,
\begin{equation*}
\nabla I_{\lambda}(\Pi_2[u_j])-\Pi_2\left[\nabla I_{\lambda}(u_j)\right]=\Pi_2\left[\mathcal{K}\left(|u_j|^{2_s^*-2}u_j-|\Pi_2\left[u_j\right]|^{2_s^*-2}\Pi_2[u_j]\right)\right].
\end{equation*}
Thus, taking $\psi\in H_{\SD}^s(\Omega)$ and writing $\psi=\Pi_1[\psi]+\Pi_2[\psi]$, we get
\begin{equation*}
\left\langle\nabla I_{\bar\lambda}(\Pi_2[u_j])-\Pi_2\left[\nabla I_{\bar\lambda}(u_j)\right],\psi\right\rangle_{H_{\SD}^s}=\left\langle\mathcal{K}\left(|u_j|^{2_s^*-2}u_j-|\Pi_2\left[u_j\right]|^{2_s^*-2}\Pi_2[u_j]\right),\Pi_2[\psi]\right\rangle_{H_{\SD}^s}.
\end{equation*}
Since
\begin{equation*}
\begin{split}
\left\langle\mathcal{K}\left(|u_j|^{2_s^*-2}u_j-|\Pi_2\left[u_j\right]|^{2_s^*-2}\Pi_2[u_j]\right),\Pi_2[\psi]\right\rangle_{H_{\SD}^s}=&\int_\Omega \left(|u_j|^{2_s^*-2}u_j-|\Pi_2\left[u_j\right]|^{2_s^*-2}\Pi_2[u_j]\right)\Pi_2[\psi] dx\\
=&(2_s^*-1)\int_\Omega|\Pi_2[u_j]|^{2_s^*-2}\Pi_1[u_j]\Pi_2[\psi] dx\\
&+O\left(\|\Pi_1[u_j]\|_{\infty}^2\right),
\end{split}
\end{equation*}
we obtain
\begin{equation*}
\begin{split}
\left|\left\langle\nabla I_{\bar\lambda}(\Pi_2[u_j])-\Pi_2\left[\nabla I_{\bar\lambda}(u_j)\right],\psi\right\rangle_{H_{\SD}^s}\right| \leq& (2_s^*-1)\int_\Omega |\Pi_2[u_j]|^{2_s^*-2}|\Pi_2[\psi]| dx\cdot\|\Pi_1[u_j]\|_{\infty}\\
& + O\left(\|\Pi_1[u_j]\|_{\infty}^2\right).
\end{split}
\end{equation*}
As a consequence of $(ii)$,
\begin{equation*}
\nabla I_{\bar\lambda}(\Pi_2[u_j])-\Pi_2\left[\nabla I_{\bar\lambda}(u_j)\right]\to 0 \quad\text{in } H_{\SD}^s
\end{equation*}
and, by \eqref{dd},
\begin{equation}\label{IlambdaprimePi2}
I_\lambda'(\Pi_2[u_j])\to 0 \quad \text{in } H^{-s}(\Omega), \quad \text{as } j\to\infty.
\end{equation}
Because of $(ii)$ and the boundedness of $\{\Pi_2[u_j]\}$, we deduce that
$$
\left\| u_j\right\|_{2^*_s}^{2^*_s}= \left\| \Pi_2[u_j]\right\|_{2^*_s}^{2^*_s} + o(1), 
$$	
and therefore we obtain
\begin{align*}
I_\lambda(u_j) & = \frac{1}{2}\left\| \Pi_1[u_j]\right\|_{H_{\SD}^s}^2 + \frac{1}{2}\left\| \Pi_2[u_j]\right\|_{H_{\SD}^s}^2 -\frac{\lambda}{2}\left\| \Pi_1[u_j]\right\|_{2}^2 - \frac{\lambda}{2}\left\| \Pi_2[u_j]\right\|_{2}^2 -\frac{1}{2^*_s}\left\| \Pi_2[u_j]\right\|_{2^*_s}^{2^*_s} +o(1) \\
& = I_\lambda(\Pi_2[u_j]) - \varrho_j,  
\end{align*}
with
$$
\varrho_j:=\frac{1}{2}\left\| \Pi_1[u_j]\right\|_{H_{\SD}^s}^2-\frac{\lambda}{2}\left\| \Pi_1[u_j]\right\|_{2}^2 +o(1)=o(1).
$$
Therefore, choosing $\varepsilon\in(0,c^*-c)$, there exists $j_0>0$ such that $|\varrho_j|\leq\varepsilon$ and 
\begin{equation}\label{IlambdaPi2}
|I_\lambda(\Pi_2[u_j])|=|I_\lambda(u_j)-\varrho_j|\leq c+\varepsilon<c^*,
\end{equation}
for any $j\geq j_0$. Because of \eqref{IlambdaprimePi2}, \eqref{IlambdaPi2} and Proposition \ref{propPS} it follows that, up to a subsequence, $\Pi_2[u_j]\to u_\infty\in H_{\SD}^s(\Omega)$. 
Then, by $(ii)$, still  up to a subsequence, $u_j\to u_\infty$ in $H_{\SD}^s(\Omega)$ and the proof is concluded.
\end{proof}	


We are now in a position to show that $I_\lambda$ satisfies the $\nabla$-condition.

\begin{proposition}\label{nablacond}
Let $k\in\N$, $k\geq 2$, and let $\Lambda_k$ be an eigenvalue of multiplicity $m\in\N$. Then, for any $\lambda\in(\Lambda_{k-1},\Lambda_{k+m})$ 
and for any $\varepsilon',\varepsilon''\in\R$ satisfying $0<\varepsilon'<\varepsilon''<\varepsilon_\lambda$, with $\varepsilon_{\lambda}$ given in \eqref{defepsilonlambda}, the functional $I_\lambda$ satisfies  $(\nabla)(I_\lambda,\mathbb V_{k,m},\varepsilon',\varepsilon'')$.
\end{proposition}

\begin{proof}

We argue by contradiction. Let us assume the existence of $\bar\lambda\in(\Lambda_{k-1},\Lambda_{k+m})$ and $\varepsilon'<\varepsilon''$ in $(0,\varepsilon_{\bar\lambda})$ for which the $\nabla$-condition $(\nabla)(I_{\bar\lambda},\mathbb V_{k,m},\varepsilon',\varepsilon'')$ does not hold.
By Lemma \ref{uniquecritpoint}, $u=0$ is the only critical point of $I_{\bar\lambda}$ on $\mathbb V_{k,m}$ with $ I_{\bar\lambda}(u)\in (-\varepsilon_{\bar\lambda},\varepsilon_{\bar\lambda})$. Let $\{u_j\}\subset H_{\SD}^s(\Omega)$ be such that
\[ I_{\bar\lambda}(u_j)\in[\varepsilon',\varepsilon'']\quad \text{ for all } j\in\N\]
so that, by \eqref{llv1}, we have $I_{\bar\lambda}(u_j)<c^*$, with $c^*$ given in \eqref{level}, and
\begin{equation}\label{Pnabla}
\text{dist}(u_j,\mathbb V_{k,m})\to 0\qquad\text{and}\qquad\Pi_2\left[\nabla I_{\bar\lambda}(u_j)\right]\to 0,
\end{equation}
as $j\to\infty$. Then, by Lemma \ref{Ujbounded}, the sequence $\{u_j\}$ strongly converges in $H_{\SD}^s(\Omega)$, namely,
\begin{equation*}
u_j\to u_\infty \text { in } H_{\SD}^s(\Omega).
\end{equation*}
In addition, $u_\infty$ turns out to be a critical point of $I_{\bar\lambda}|_{\mathbb{V}_{k,m}}$. Indeed, by \eqref{Pnabla}, we get
\begin{align*}
	\left\langle I'_{\bar\lambda}(u_j),\psi\right\rangle_{H^{-s}} & = \left\langle u_j,\psi\right\rangle_{H_{\SD}^s} -\bar \lambda\int_\Omega u_j(x)\psi(x) dx -\int_\Omega |u_j|^{2_s^*-2}u_j\psi(x)dx \to 0
\end{align*}
as $j\to\infty$, for all $\psi\in \mathbb{V}_{k,m}$. Hence, using again Lemma \ref{Ujbounded}, we conclude
\[
\left\langle I'_{\bar\lambda}(u_\infty),\psi\right\rangle_{H^{-s}} = \left\langle u_\infty,\psi\right\rangle_{H_{\SD}^s} -\bar\lambda\int_\Omega u_\infty(x)\psi(x) dx -\int_\Omega |u_\infty|^{2_s^*-2}u_\infty\psi(x)dx.
\]
So, it must be $u_\infty=0$. Since $I_{\bar\lambda}(u_j)\geq\varepsilon'>0$, we get $I_{\bar\lambda}(u_\infty)>0$, a contradiction.
\end{proof}

\section{Proof of Theorem \ref{mainresult}}\label{proofofmainresult}

\begin{lemma}\label{lemaa}
Let $k\in\N$, $k\geq 2$, and let $\Lambda_k$ be an eigenvalue of multiplicity $m\in\N$. Then, for any $0<\delta<\Lambda_k$ and any $\lambda\in(\Lambda_{k}-\delta, \Lambda_{k})$, one has
\begin{equation*}
	\sup_{u\in \mathbb{H}_{k+m-1}}I_\lambda(u)<\frac{s}{N}|\Omega|\delta^\frac{N}{2s}:=\eta_\delta.
\end{equation*}
In particular, 
\begin{equation*}
	\lim_{\lambda\to\Lambda_k}\sup_{u\in \mathbb{H}_{k+m-1}}I_\lambda(u)=0.
\end{equation*}
\end{lemma}

\begin{proof}

By contradiction, let $\bar\delta>0$ and $\bar\lambda\in(\Lambda_k-\bar\delta,\Lambda_k)$ such that
\begin{equation}\label{eqd}
\sup\limits_{u\in \mathbb{H}_{k+m-1}}I_{\bar\lambda}(u)\geq\eta_{\bar\delta}.
\end{equation}
Let also $u_0\in \mathbb{H}_{k+m-1}$ be such that
$$
I_{\bar\lambda}(u_0)= \sup\limits_{u\in \mathbb{H}_{k+m-1}}I_{\bar\lambda}(u).
$$
By Lemma \ref{varcharacteigen} and \eqref{eqd}, we derive
\begin{equation*}
\begin{split}
\eta_{\bar\delta}\leq I_{\bar\lambda}(u_0)&=\frac{1}{2}\left\| u_0\right\|_{H_{\SD}^s}^2-\frac{\bar\lambda}{2}\|u_0\|_2^2-\frac{1}{2_s^*}\|u_0\|_{2_s^*}^{2_s^*}\\
&\leq\frac12(\Lambda_{k}-\bar\lambda)\|u_0\|_2^2-\frac{1}{2_s^*}\|u_0\|_{2_s^*}^{2_s^*}\\
&\leq\frac12(\Lambda_{k}-\bar\lambda)\|u_0\|_2^2-\frac{1}{2_s^*} \left( |\Omega|^{-\frac{2s}{N}} \|u_0\|_2^2\right)^{\frac{2_s^*}{2}}\\
&=\frac12(\Lambda_{k}-\bar\lambda)\|u_0\|_2^2-\frac{1}{2_s^*}|\Omega|\left(\frac{\|u_0\|_2^2}{|\Omega|}\right)^{\frac{2_s^*}{2}} \\
&=\frac12(\Lambda_{k}-\bar\lambda)|\Omega|\left(\frac{\|u_0\|_2^2}{|\Omega|}\right)-\frac{1}{2_s^*}|\Omega|\left(\frac{\|u_0\|_2^2}{|\Omega|}\right)^{\frac{2_s^*}{2}}.
\end{split}
\end{equation*}
Now observe that, for $\tau>0$, the function
\begin{equation*}
h(\tau):=\frac12(\Lambda_{k}-\bar\lambda)\tau-\frac{1}{2_s^*}\tau^{\frac{2_s^*}{2}},
\end{equation*}
attains its maximum at $\tau_0:=(\Lambda_k-\bar\lambda)^{\frac{2}{2_s^*-2}}$ and one has
\begin{equation*}
h(\tau_0)=\frac{s}{N}(\Lambda_k-\bar\lambda)^{\frac{N}{2s}}.
\end{equation*}
Therefore, we obtain
\begin{equation*}
\eta_{\bar\delta} \leq \frac{s}{N}|\Omega|(\Lambda_k-\bar\lambda)^{\frac{N}{2s}}<\frac{s}{N}|\Omega|{\bar\delta}^{\frac{N}{2s}}=\eta_{\bar\delta},
\end{equation*}
a contradiction.
\end{proof}

\begin{proof}[Proof of Theorem \ref{mainresult}]
Let $\Lambda_k$, $k\geq 2$, be an eigenvalue of multiplicity $m\in\N$. By Proposition \ref{nablacond}, for any $\lambda\in(\Lambda_{k-1},\Lambda_{k+m})$, it follows that $(\nabla)(I_\lambda,\mathbb{H}_{k-1}\oplus \mathbb{P}_{k+m-1},\varepsilon',\varepsilon'')$ holds for all $\varepsilon',\varepsilon''\in\R$ satisfying  $0<\varepsilon'<\varepsilon''<\varepsilon_\lambda$.
	
Let us show now that there exists $\bar\delta_k\in(0,\Lambda_k)$ such that, for any $\lambda\in\left(\Lambda_k-\bar\delta_k,\Lambda_k\right)$, one has
\begin{equation}\label{etavarepsilon}
\eta_{\bar\delta_k} < \varepsilon_\lambda.	
\end{equation}

We distinguish two cases. If $\lambda<H_k^*$, then
\eqref{etavarepsilon} is equivalent to the system
\begin{equation}\label{system1}
\left\lbrace 
\begin{array}{l}
\Lambda_k -\bar\delta_k <\lambda<\Lambda_k,\smallskip\\
\displaystyle\bar\delta_k<\frac{\lambda-\Lambda_{k-1}}{\Lambda_{k-1}}|\Omega|^{-\frac{2s}{N}}\widetilde{S}(\Sigma_{\mathcal{D}}),\smallskip\\
0<\bar\delta_k<\Lambda_k.	
\end{array}
\right. 
\end{equation}
By direct computations it is easily seen that the admissible value of $\bar\delta_k$ which maximizes the range of $\lambda$ is
$$
\delta_k^{(1)}:=\frac{\Lambda_k-\Lambda_{k-1}}{\Lambda_{k-1}+\widetilde{S}(\Sigma_{\mathcal{D}})|\Omega|^{-\frac{2s}{N}}}|\Omega|^{-\frac{2s}{N}}\widetilde{S}(\Sigma_{\mathcal{D}}).
$$
Arguing similarly, when $\lambda>H_k^*$, \eqref{etavarepsilon} amounts to
\begin{equation}\label{system2}
	\left\lbrace 
	\begin{array}{l}
		\Lambda_k -\bar\delta_k <\lambda<\Lambda_k,\smallskip\\
		\displaystyle\bar\delta_k<\frac{\Lambda_{k+m}-\lambda}{\lambda_{k+m}}|\Omega|^{-\frac{2s}{N}}\widetilde{S}(\Sigma_{\mathcal{D}}),\smallskip\\
		0<\bar\delta_k<\Lambda_k,	
	\end{array}
	\right. 
\end{equation}
which has
$$
\delta_k^{(2)}:=\frac{\Lambda_{k+m}-\Lambda_{k}}{\Lambda_{k+m}}|\Omega|^{-\frac{2s}{N}}\widetilde{S}(\Sigma_{\mathcal{D}})
$$
as an extremal value for $\bar\delta_k$.

Therefore, choosing $\delta_k$ as in \eqref{defdeltak} and taking also account of Lemma \ref{lemaa}, we obtain
\begin{equation}
\sup_{u\in \mathbb{H}_{k+m-1}}I_\lambda(u) < \eta_{\delta_k} <\varepsilon_\lambda,
\end{equation} 
for all $\lambda\in(\Lambda_k-\delta_k,\Lambda_k)$. Letting $\varepsilon''\in \left(\sup_{ \mathbb{H}_{k+m-1}}I_\lambda,\eta_{\delta_k}\right)$ and taking also account of Propositions \ref{prop_supinf} and \ref{propPS}, by Theorem \ref{nablathm}, $I_\lambda$ has two critical points $u_1$ and $u_2$ satisfying $I_\lambda(u_i)\in [\varepsilon',\varepsilon'']$. In particular, being $\varepsilon''$ arbitrary we obtain that
\begin{equation}\label{energiesmallersup}
	0<I_\lambda(u_i) \leq \sup_{u\in \mathbb{H}_{k+m-1}}I_\lambda(u),
\end{equation}
for every $\lambda\in(\Lambda_k-\delta_k,\Lambda_k)$.
Then, Theorem \ref{mainresult} is proved.
\end{proof}

\section*{Acknowledgement}
\noindent G. Molica Bisci is funded by the European Union - NextGenerationEU within the framework of PNRR Mission 4 - Component 2 - Investment 1.1 under the Italian Ministry of University and Research (MUR) program PRIN 2022 - grant number 2022BCFHN2 - Advanced
theoretical aspects in PDEs and their applications - CUP: H53D23001960006.\\
A. Ortega is partially funded by Vicerrectorado de Investigación, Transferencia del Conocimiento y Divulgación Científica of Universidad Nacional de Educación a Distancia under research project Talento Joven UNED 2025, Ref: 2025/00151/001.\\
L. Vilasi is a member of the Gruppo Nazionale per l'Analisi Matematica, la Probabilità e le loro Applicazioni (GNAMPA) of the Istituto Nazionale di Alta Matematica (INdAM). This work is partially funded by the ``INdAM - GNAMPA Project CUP E5324001950001".



\begin{thebibliography}{99}

\bibitem{mu0} L. Appolloni, G. Molica Bisci and S. Secchi, \textit{Multiple solutions for Schrödinger equations on Riemannian manifolds via $\nabla$-theorems}, Ann. Global Anal. Geom. {\bf 63} (2023), no. 1, 22 pp.

\bibitem{Barrios2020} B. Barrios and M. Medina, {\it Strong maximum principles for fractional elliptic and parabolic problems with mixed boundary conditions}, Proc. Roy. Soc. Edinburgh Sect. A {\bf 150}, No 1 (2020), 475--495.

\bibitem{Bartolo1983} P. Bartolo, V. Benci and D. Fortunato, \textit{Abstract critical point theorems and applications to some
nonlinear problems with strong resonance at infinity}, Nonlinear Anal. \textbf{7} (1983) 981--1012.

\bibitem{BonforteSirVaz2015}M. Bonforte, Y. Sire, J.L. Vázquez \textit{Existence, uniqueness and asymptotic behaviour for fractional porous medium equations on bounded domains} Discrete Contin. Dyn. Syst. \textbf{35} (2015), no. 12, 5725--5767.

\bibitem{bracoldepsan2013a} C. Brändle, E. Colorado, A. de Pablo and U. Sánchez, \textit{A concave--convex elliptic
problem involving the fractional Laplacian}, Proc. Roy. Soc. Edinburgh Sect. A \textbf{143} (2013) 39--71.

\bibitem{Carmona2020a}
J. Carmona, E. Colorado, T. Leonori and A. Ortega, \textit{Semilinear fractional elliptic problems with mixed Dirichlet-Neumann boundary conditions},  Fract. Calc. Appl. Anal. \textbf{23}(4), 1208--1239 (2020). DOI: 10.1515/fca-2020-0061

\bibitem{carcolleoort2020} J. Carmona, E. Colorado, T. Leonori and A. Ortega, {\it Regularity of solutions to a fractional elliptic problem with mixed Dirichlet--Neumann boundary data}, Adv. Calc. Var. {\bf 14} (4) (2021), 521--539.

\bibitem{colort2019} E. Colorado and A. Ortega, \textit{The Brezis-Nirenberg problem for the fractional Laplacian with mixed Dirichlet-Neumann boundary conditions}, J. Math. Anal. Appl. {\bf 473} (2019), no. 2, 1002--1025.


\bibitem{Fiscella2016}
A. Fiscella, G. Molica Bisci and R. Servadei, {\em Bifurcation and multiplicity results for critical nonlocal fractional Laplacian problems}, Bull. Sci. math. {\bf }140 (2016), 14--35.

\bibitem{Leonori2018} T. Leonori, M. Medina, I. Peral, A. Primo and F. Soria, {\it Principal eigenvalue of mixed problem for the fractional Laplacian: Moving the boundary conditions} J. Differential Equations {\bf 265} (2) (2018), 593--619.

\bibitem{Lions1972}
J.-L. Lions and E. Magenes, {\em Non-homogeneous boundary value problems and applications. Vol. I}, Springer-Verlag, New York-Heidelberg, 1972.
\newblock Translated from the French by P. Kenneth, Die Grundlehren der mathematischen Wissenschaften, Band 181. Springer-Verlag, New York-Heidelberg, 1972. xvi+357 pp.

\bibitem{LopezOrtega2021}
R. L\'opez-Soriano and A. Ortega. \textit{A strong maximum principle for the fractional Laplace equation with
mixed boundary condition}. Fract. Calc. Appl. Anal. \textbf{24}(6), 1699--1715 (2021).

\bibitem{mms} {P. Magrone, D. Mugnai and R. Servadei}, \emph{Multiplicity of solutions for semilinear variational inequalities via linking and $\nabla-$theorems}, J. Differential Equations {\bf 228} (2006), 191--225.

\bibitem{marsac1997some} A. Marino and C. Saccon, \textit{Some variational theorems of mixed type and elliptic problems with jumping nonlinearities}, Ann. Scuola Norm. Sup. Pisa Cl. Sci. {\bf 25} (4) (1997), 631--665.

\bibitem{mu2} G. Molica Bisci, D. Mugnai and R. Servadei, \textit{On multiple solutions for nonlocal fractional problems via $\nabla$-theorems}, Differential Integral Equations {\bf 30} (2017), no. 9-10, 641--666.

\bibitem{mov2023} G. Molica Bisci, A. Ortega, L. Vilasi, \emph{Subcritical nonlocal problems with mixed boundary conditions}, Bull. Math. Sci., 2350011, DOI: 10.1142/S166436072350011X, 23 pp., 2023.

\bibitem{MRS}{G. Molica Bisci, V. R\u{a}dulescu and R. Servadei},
Variational Methods for Nonlocal Fractional Problems. With a Foreword by Jean Mawhin, {\em Encyclopedia of Mathematics and its Applications}, {\em Cambridge University Press} \textbf{162}, Cambridge, 2016. ISBN
9781107111943.

\bibitem{mu1} D. Mugnai, \textit{Multiplicity of critical points in presence of a linking: application to a superlinear boundary value problem}, NoDEA Nonlinear Differential Equations Appl. {\bf 11} (2004), no. 3, 379--391.

\bibitem{mu4} D. Mugnai, \emph{Four nontrivial solutions for subcritical exponential equations}, Calc. Var. Partial Differential Equations, {\bf 32} (2008), 481--497.

\bibitem{Mukherjee2024} T. Mukherjee, P. Pucci and L. Sharma, {\em Nonlocal critical exponent singular problems under mixed Dirichlet-Neumann boundary conditions}, J. Math. Appl. \textbf{531} (2024), Paper no. 127843, 28pp.

\bibitem{Mukherjee2025} T. Mukherjee and L. Sharma, \textit{On elliptic problems with mixed operators and Dirichlet-Neumann boundary conditions}, Nonlinear Differ. Equ. Appl. \textbf{32}, 80 (2025).

\bibitem{Ortega2023} A. Ortega, {\it Concave-convex critical problems for the spectral fractional laplacian with mixed boundary conditions}. Fract. Calc. Appl. Anal. {\bf 26}, 305--335 (2023). 


\end{thebibliography}
\end{document}